\documentclass[12pt]{amsart}
\usepackage[utf8]{inputenc}
\usepackage{amsmath}
\usepackage{amsfonts}
\usepackage{amssymb}
\usepackage{amsthm}
\usepackage{bbm}
\usepackage{hyperref}
\renewcommand{\sqrt}[1]{\left( #1 \right)^\frac12}

\newcommand{\E}[0]{\mathbb{E}}

\newcommand{\N}[0]{\mathbb{N}}

\newcommand{\R}[0]{\mathbb{R}}
\newcommand{\Z}[0]{\mathbb{Z}}

\renewcommand{\d}[0]{\mathrm{d}}
\newcommand{\sgn}[0]{\mathrm{sgn}}

\newcommand{\p}[1]{\left(#1\right)}

\newcommand{\beq}{\begin{equation}}
\newcommand{\eeq}{\end{equation}}
\newcommand{\supp}[0]{\mathrm{supp}\,}

\author{Maciej Rzeszut}
\email{maciej.rzeszut@gmail.com}

\title{Some remarks on Davis inequality for biparameter filtrations}
\newtheorem{thm}{Theorem}
\newtheorem{lem}[thm]{Lemma}
\newtheorem{deff}[thm]{Definition}
\newtheorem{cor}[thm]{Corollary}

\numberwithin{thm}{section}
\numberwithin{equation}{section}

\allowdisplaybreaks
\begin{document}
\begin{abstract}The Davis inequality $\E Sf\simeq \E f^*$ between $L^1$ norms of square function of a martingale and its maximal function is known for martingales indexed by linearly ordered filtrations and in some particular cases for double indexed one. We prove the $\gtrsim $ inequality for arbitrary filtrations satisfying the (F4) condition of Cairoli and Walsh and propose a method to attack the other inequality. The former is done by means of a two-parameter analogue of Davis-Garsia decomposition.\end{abstract}
\maketitle
\section{Introduction}
Let us briefly introduce the notation and some known results concerning martingale $H^1$ spaces in one and two parameters. The reader may find a more detailed exposition in \cite{niestety} and references within.\par
Let $\p{\mathcal{F}_i}_i$ be a filtration indexed by $i\in \N=\N\cup\{0\} $ or $i\in \N^2$. The set $\N^2$ will be equipped with partial order $(i,j)\leq \p{i',j'}\iff i\leq i',j\leq j'$ and $(i,j)-1$ will stand for $(i-1,j-1)$. For brevity we will write $\E_i= \E\p{\cdot\mid \mathcal{F}_i}$. The martingale difference operators are defined by 
\beq\Delta_i=\E_i-\E_{i-1}\text{ for }i\in \N_+,\quad \Delta_0=\E_0;\eeq
\beq \Delta_{i,j}=\E_{i,j}+\E_{i-1,j-1}-\E_{i-1,j}-\E_{i,j-1},\eeq
\beq \Delta_{i,0}= \E_{i,0}-\E_{i-1,0},\quad \Delta_{0,j}=\E_{0,j}-\E_{0,j-1}\text{ for }i,j\in\N^2_+,\eeq
\beq \Delta_{0,0}=\E_{0,0}.\eeq
The definition of $\Delta_{i,j}$ with indices containing 0 is made to ensure
\beq f=\sum_{i,j\geq 0}\Delta_{i,j}f.\eeq
For $i,j\in \Z$, we put 
\beq \E_i=\E_{i\vee 0}, \quad \E_{i,j}=\E_{i\vee 0,j\vee 0}.\eeq
The square function, conditional square function and maximal function [with respect to $\p{\mathcal{F}_i}$] are defined by
\beq S_p f=\left(\sum_{i\geq 0} \left|\Delta_i f\right|^p\right)^\frac1p,\quad s_p f= \left(\sum_{i\geq 0} \E_{i-1}\left|\Delta_i f\right|^p\right)^\frac1p, \quad Mf= \sup_{i\geq 0} \left| \E_i f\right|\eeq
and the corresponding Hardy space norms by 
\beq \|f\|_{H_1^S}= \|f\|_{S_2 L_1}=\E S_2 f,\quad \|f\|_{H^s_1}= \|f\|_{s_2 L_1}= \E sf,\quad \|f\|_{H_1^*}= \|f\|_{ML_1}=\E Mf.\eeq\par
A double-indexed filtration is said to satisfy the (F4) condition if $\bigvee_{k}\mathcal{F}_{i,k}$ and $\bigvee_k \mathcal{F}_{k,j}$ are conditionally independent w.r.t. $\mathcal{F}_{i,j}$. A filtration is called regular if, for any positive $f$, $\E_{i+1}f\lesssim \E_i f$ (one-parameter case) and $\E_{i+1,j}f,\E_{i,j+1}f\lesssim \E_{i,j} f$ (two-parameter case). \par
Let $\p{\Omega,\mathcal{F},\mu}$ be a probability space. A canonical one-parameter filtration of product type is a filtration $\p{\mathcal{F}_i}_{i\in \N}$ on $\Omega^\N$ such that $\mathcal{F}_i$ is generated by the coordinate projection $x\mapsto \p{x_k}_{k\leq i}$. A canonical two-parameter filtration of product type is a filtration $\p{\mathcal{F}_{i,j}}_{i,j\in\N}$ on $\Omega^\N\times \Omega^\N$ such that $\mathcal{F}_{i,j}=\mathcal{F}_i\otimes \mathcal{F}_j$ and $\p{\mathcal{F}_i}_{i\in\N}$ is a canonical filtration of product type on $\Omega^\N$. These filtrations are regular iff $\Omega$ is finite, with a constant depending on the smallest mass of an atom of $\Omega$. The two-parameter one always satisfies the (F4) condition. However, this model is not universal and is characterized by the property 
\beq \mathcal{F}_{i-1,j}\vee\mathcal{F}_{i,j-1}=\mathcal{F}_{i,j}.\eeq
In \cite{RzT}, we prove that when we restrict our attention to the case of discrete probability spaces, a universal model is 
\beq\label{eq:univdefi} \bigotimes_{i,j}\p{\Omega_{i,j},\mathcal{S}_{i,j},\mu_{i,j}},\eeq
where $\mathcal{F}_{i,j}$ is defined by 
\beq \label{eq:univdefii}\mathcal{F}_{i,j}= \sigma\p{x\mapsto \p{x_{k,l}}_{k\leq i, l\leq j}}.\eeq \par

For any two Banach spaces $X$, $Y$ such that adding their elements is well defined, by $X+Y$ we mean their algebraic sum with the norm
\beq\|f\|_{X+Y}=\inf_{\substack{f=g+h\\ g\in X,h\in Y}}\|g\|_{X}+\|h\|_{Y}.\eeq
In one-parameter case, a theorem of Davis \cite{davis} states that 
\beq H_1^S=H_1^*.\eeq
Moreover, it is known \cite{garsia} that 
\beq H_1^S = H_1^s + S_1 L_1\eeq
(to which we refer as the Davis-Garsia decomposition). In particular, $\|f\|_{H_1^s}\geq \|f\|_{H_1^S}$. Also, if the filtration is regular, all three norms are equivalent: $\|f\|_{H_1^S}\simeq \|f\|_{H_1^s}\simeq \|f\|_{H_1^*}$. \par
In two-parameter case, it is only known that
\beq \|f\|_{H_1^s}\gtrsim \|f\|_{H_1^*},\eeq
\beq \|f\|_{H_1^s}\gtrsim \|f\|_{H_1^S},\eeq
provided the filtration satisfies (F4). Both inequalities are due to Weisz \cite{niestety}. If additionally the filtration is regular, $H_1^S=H_1^*=H_1^s$. The question whether $H_1^S=H_1^*$ is open without the regularity assumption. We are going to prove one inclusion in the canonical product type case, by decomposing $H_1^S$ into 4 summands, each of which is contained in $H_1^*$, then indicate how to perform an analogous proof in the universal case. \par

\section{Decoupling inequalities and 4 summand decomposition}
We will use the following decoupling inequality \cite{zinn}.
\begin{thm}[Zinn]Let $f_k$ be a sequence of fucntions adapted to the canonical product filtration on $\Omega^\N$. Then, treating each $f_k$ as a function on $\Omega^k$,
\beq \int_{\Omega^\N}\sqrt{\sum_k f_k\p{x_1,\ldots,x_k}^2}\d x\simeq \int_{\Omega^\N}\int_{\Omega^\N}\sqrt{\sum_k f_k\p{x_1,\ldots,x_{k-1},\eta_k}^2}\d x\d \eta.\eeq\end{thm}
\begin{cor}\label{makeitdouble}Let $f_{i,j}$ be a sequence of functions adapted to the canonical product filtration on $\Omega^\N\times \Omega^\N$. Then, treating $f_{i,j}$ as a function on $\Omega^i\times \Omega^j$,
\beq \int_{\Omega^\N\times \Omega^\N} \sqrt{\sum_{i,j} f_{i,j}\p{x_{\leq i},y_{\leq j}}^2}\d x\d y\simeq\eeq
\beq \int_{\Omega^\N\times \Omega^\N} \int_{\Omega^\N}\int_{\Omega^\N}\sqrt{\sum_{i,j} f_{i,j}\p{\p{x_{<i},\eta_i},\p{y_{< j},\theta_j}}^2}\d \eta\d\theta\d x\d y.\eeq\end{cor}
We also recall our previous result \cite{diss}. 
\begin{thm}\label{4summand} Let $f_{i,j}\in L^1\p{\Omega^2}$. There exist $\{0,1\}$-valued functions $a_{i,j},b_{i,j},c_{i,j},d_{i,j}$ on $\Omega^2$ such that
\beq a_{i,j}+b_{i,j}+c_{i,j}+d_{i,j}=1\eeq
and
\begin{align}
\lefteqn{\int_{\Omega^\N}\int_{\Omega^\N}\sqrt{\sum_{i,j} f_{i,j}\p{x_i,y_j}^2}\d x\d y}\\
\gtrsim & \sum_{i,j}\int_\Omega\int_\Omega \left|a_{i,j}f_{i,j}\right|\p{x_i,y_j}\d x_i \d y_j\\
+& \sqrt{ \sum_{i,j}\int_\Omega\int_\Omega \left(b_{i,j}f_{i,j}\right)^2\p{x_i,y_j}\d x_i \d y_j}\\
+& \sum_i \int_\Omega \sqrt{\sum_j \int_\Omega \p{c_{i,j}f_{i,j}}\p{x_i,y_j}^2 \d y_j}\d x_i\\
+& \sum_j \int_\Omega \sqrt{\sum_i \int_\Omega \p{d_{i,j}f_{i,j}}\p{x_i,y_j}^2 \d x_i}\d y_j.
\end{align}
\end{thm}
Here, equpping $x_i$'s and $y_i$' with changing indices on the right hand side is purely ornamental and we could write $\int_{\Omega}\d\xi$, $\int_{\Omega}\d\upsilon$ instead od $\int_{\Omega}\d x_i$, $\int_{\Omega}\d y_j$.

\section{Main lemma}
We are going to show an improved version of the main lemma used in the proof of Theorem \ref{4summand}. 
\begin{deff}A (not necessarily ordered) family of sigma-algebras $\mathcal{F}_i$ is said to have $p$-Doob property if the inequality
\beq \delta\left\| \sup_i \left|\E_i f\right|\right\|_{L_p}\leq\|f\|_{L_p}\eeq
is satisfied for some constant $\delta>0$ and all $f$.\end{deff}
It is known that (F4) filtrations have $p$-Doob property for any $p\in (1,\infty)$. For a canonical product filtration, this can be checked by iterating one-parameter Doob inequality.
\begin{lem}\label{westvirginia}Let $\p{\Omega,\mathcal{F},\mu}$ be a finite probability space, $\p{\mathcal{F}_i}_{i\in I}$ be a filtration on $\Omega$ indexed by a finite set $I$, which has $2$-Doob property with constant $\delta$. Let also $J$ be a finite set and $w_{i,j}$ for $(i,j)\in I\times J$ be $[0,1]$-valued functions on $\Omega$. Then, for any $\kappa>0$ and any sequence $\p{f_{i,j}}_{(i,j)\in I\times J}$ of functions such that $f_{i,j}$ is $\mathcal{F}_i$-measurabe,
\beq \label{eq:lemmainineq}\E\sqrt{\sum_{(i,j)\in I\times J} \p{w_{i,j}f_{i,j}}^2}\geq \kappa^2 \delta \E\sqrt{\sum_{(i,j)\in I\times J} \mathbbm{1}_{\left\{\E_i w_{i,j} \geq \kappa\right\}}f_{i,j}^2}.\eeq
\end{lem}
\begin{proof}
Since $\Omega,I,J$ are finite, we are in a finite dimensional setting. Let us define 
\beq  f_i=\p{f_{i,j}}_{j\in J}\in L_2\p{\ell_2\p{J}},\quad f=\p{f_i}_{i\in I}\in L_2\p{\ell_2\p{I,\ell_2\p{J}}}.\eeq
By $W_i$ and $W$ we will denote coordinatewise multiplication operators respectively on $L_2\p{\ell_2\p{J}}$ and $L_2\p{\ell_2\p{I,\ell_2\p{J}}}$ defined by 
\beq W_i \p{\varphi_j}_{j\in J}= \p{w_{i,j}\varphi_j}_{j \in J},\quad W \p{\varphi_i}_{i \in I}= \p{W_i \varphi_i}_{i \in I}.\eeq
Let also
\beq A_{i,j}=\left\{\E_i w_{i,j}\geq \kappa\right\},\eeq
\beq\mathcal{A}=\left\{\varphi:\Omega\to\p{\ell_2\p{I\times J}}: \supp\varphi_{i,j}\subset A_{i,j}\right\},\eeq
\beq h_p\p{\ell_2\p{I\times J}}= \left\{\varphi\in L_p\p{\ell_2\p{I\times J}}: \varphi_i\text{ is }\mathcal{F}_i\text{-measurable}\right\},\eeq
\beq h_p^\mathcal{A}\p{\ell_2\p{I\times J}}=\left\{\varphi\in h_p\p{\ell_2\p{I\times J}}: \varphi\in\mathcal{A}\right\}\eeq
For any Banach space norm $B$ on functions $\Omega\to \ell_2\p{I\times J}$, we will denote by $B^*$ its dual, w.r.t. the pairing
\beq \langle \varphi,\psi\rangle = \sum_i \E \langle \varphi_i,\psi_i\rangle_{\ell_2(J)}.\eeq\par
We can without loss of generality assume $f\in \mathcal{A}$, because replacing $f_{i,j}$ by $\mathbbm{1}_{A_{i,j}}$ preserves being adapted to $\mathcal{F}_{i}$ and decreases the left hand side of \eqref{eq:lemmainineq}. Since $h_1^\mathcal{A}\p{\ell_2\p{I\times J}}$ is finite dimensional, by reflexivity we have 
\beq \|\varphi\|_{h_1^\mathcal{A}\p{\ell_2\p{I\times J}}}= \sup_{\|\psi\|_{\p{h_1^\mathcal{A}\p{\ell_2\p{I\times J}}}^*}\leq 1}\left|\langle\varphi,\psi\rangle\right|.\eeq
Therefore, in order to prove the lemma, it is enough to fix an element $g$ of the unit sphere of $\p{h_1^\mathcal{A}\p{\ell_2\p{I\times J}}}^*$, represented as an adapted sequence of functions belonging to $\mathcal{A}$, and prove
\beq \E\sqrt{\sum_i \left\|W_i f_i\right\|_{\ell_2(J)}^2}\geq \kappa^2\delta \langle f,g\rangle. \eeq
Without loss of generality, $\langle f,g\rangle =1$. The left hand side, as a function on the affine space $\mathcal{A}_1:=\left\{f\in \mathcal{A}: \langle f,g\rangle =1\right\}$ is convex and goes to $\infty$ when $f\to \infty$. Indeed,
\begin{align} \Phi(f):= & \E\sqrt{\sum_i \left\|W_i f_i\right\|_{\ell_2(J)}^2}\\
\geq & \sqrt{\sum_{i,j} \left(\E \left|w_{i,j}\right|\cdot \left|f_{i,j}\right|\right)^2}\\
\geq & \kappa \sqrt{\sum_{i,j}\left( \E \left|f_{i,j}\right|\right)^2}\\
=& \kappa \|f\|_{\ell_2\p{I\times J, L_1}},\end{align}
which is a norm on $\mathcal{A}$. Therefore, the infimum of $\Phi(f)$ over $f\in\mathcal{A}_1$ is attained at some $f$. One easily calculates the gradient of $\Phi$ as a function on $L_2\p{\ell_2\p{I\times J}}$ to be
\beq \nabla_{L^2}\Phi(f)= \p{\frac{W_i^2 f_i}{\|Wf\|_{\ell_2(I\times J)}}}_{i\in I}.\eeq
We obtain the gradient on $h_2$ by projecting onto sequences adapted to $\p{\mathcal{F}_i}_{i\in I}$ in the first index. Since $\E_i w_{i,j}^2f_{i,j}= f_{i,j} \E_{i}w_{i,j}^2$ is already supported on $A_{i,j}$, we get 
\beq \nabla_{h_2^\mathcal{A}}\Phi(f)= \p{\E_i\frac{W_i^2 f_i}{\|Wf\|_{\ell_2(I\times J)}}}_{i\in I}.\label{eq:gradh2a}\eeq
If $f$ is the minimizer or $\Phi$ on $\mathcal{A}_1$, there is a scalar $\lambda$ such that
\beq \nabla_{h_2^\mathcal{A}}\Phi(f)=\lambda g,\eeq
where the equality is to be understood as equality of functionals on $h_2^\mathcal{A}\p{\ell_2(I\times J)}$. But \eqref{eq:gradh2a} gives a representation of $\nabla_{h_2^\mathcal{A}}\Phi(f)$ as a sequence of functions which belongs to $\mathcal{A}$ and is adapted, so
\beq \p{\E_i\frac{W_i^2 f_i}{\|Wf\|_{\ell_2(I\times J)}}}_{i\in I}=\lambda g.\label{eq:lemlagr}\eeq
Therefore 
\begin{align} \Phi(f) =& \E\sum_i \frac{\left\|W_i f_i\right\|_{\ell_2(J)}^2}{\left\|Wf\right\|_{\ell_2(I\times J)}}\\
=& \sum_i \E \left\langle f_i, \frac{W_i^2 f_i}{\left\|Wf\right\|_{\ell_2(I\times J)}}\right\rangle_{\ell_2(J)}\\
=& \sum_i \E \left\langle f_i, \E_i\frac{W_i^2 f_i}{\left\|Wf\right\|_{\ell_2(I\times J)}}\right\rangle_{\ell_2(J)}\\
=&\lambda \langle f,g\rangle =\lambda.\end{align}
In particular, $\lambda \geq 0$ and we only have to prove $\lambda \geq \kappa^2\delta$. But by \eqref{eq:lemlagr} and $\|g\|_{\p{h_1^\mathcal{A}\p{\ell_2\p{I\times J}}}^*}=1$, this is equivalent to 
\beq \label{eq:almostheaven}\left\|\p{\E_i\frac{W_i^2 f_i}{\|Wf\|_{\ell_2(I\times J)}}}_{i\in I}\right\|_{\p{h_1^\mathcal{A}\p{\ell_2\p{I\times J}}}^*}\geq \kappa^2\delta.\eeq
By Cauchy-Schwarz,
\beq \left|f_{i,j}\right| \E_i \|Wf\|_{\ell_2(I\times J)} \E_i \frac{w_{i,j}^2}{\|Wf\|_{\ell_2(I\times J)}} \geq  \left|f_{i,j}\right|\p{\E_{i}w_{i,j}}^2\geq \kappa^2 \left|f_{i,j}\right|,\eeq
because the last inequality holds on $A_{i,j}$ and outside of $A_{i,j}$ both sides are 0. Therefore
\begin{align} \lefteqn{\left\|\p{\E_i\frac{W_i^2 f_i}{\|Wf\|_{\ell_2(I\times J)}}}_{i\in I}\right\|_{\p{h_1^\mathcal{A}\p{\ell_2\p{I\times J}}}^*}}\\
&= \sup_{u\in\mathcal{A}\text{ adapted}}\frac{\left|\sum_i \E \left\langle u_i,\E_i\frac{W_i^2 f_i}{\|Wf\|_{\ell_2(I\times J)}}\right\rangle_{\ell_2(J)}\right|}{\E\sqrt{\sum_i \left\|u_i\right\|_{\ell_2(J)}^2}} \\
&\geq\sup_{u\in\mathcal{A}\text{ adapted}}\frac{\left|\sum_i \E \left\langle u_i \E_{i}\|Wf\|_{\ell_2(I\times J)},\E_i\frac{W_i^2 f_i}{\|Wf\|_{\ell_2(I\times J)}}\right\rangle_{\ell_2(J)}\right|}{\E\sqrt{\sum_i \left\|u_i\E_i \|Wf\|_{\ell_2(I\times J)}\right\|_{\ell_2(J)}^2}} \\
&\geq\sup_{u\in\mathcal{A}\text{ adapted}}\frac{\left|\sum_i \E \left\langle u_i, \E_{i}\|Wf\|_{\ell_2(I\times J)}\E_i\frac{W_i^2 f_i}{\|Wf\|_{\ell_2(I\times J)}}\right\rangle_{\ell_2(J)}\right|}{\E\p{\sqrt{\sum_i \left\|u_i \right\|_{\ell_2(J)}^2}M \|Wf\|_{\ell_2(I\times J)}}} \\
&\geq\sup_{u\in\mathcal{A}\text{ adapted}}\frac{\left|\sum_i \E \left\langle u_i, \E_{i}\|Wf\|_{\ell_2(I\times J)}\E_i\frac{W_i^2 f_i}{\|Wf\|_{\ell_2(I\times J)}}\right\rangle_{\ell_2(J)}\right|}{\|u\|_{h_2^{\mathcal{A}}\p{\ell_2\p{I\times J}}} \left\|M \|Wf\|_{\ell_2(I\times J)} \right\|_{L_2}} \\
&\geq \delta \frac{\left\|\p{ \E_{i}\|Wf\|_{\ell_2(I\times J)}\E_i\frac{W_i^2 f_i}{\|Wf\|_{\ell_2(I\times J)}}}_{i\in I} \right\|_{h_2^\mathcal{A}\p{\ell_2(I\times J)}}}{\|Wf\|_{L_2\p{\ell_2(I\times J)}}}\\
&\geq \kappa^2\delta \frac{\|f\|_{h_2^\mathcal{A}\p{\ell_2(I\times J)}}}{\|Wf\|_{L_2\p{\ell_2(I\times J)}}}\\
&\geq \kappa^2\delta.
\end{align}
\end{proof}

\section{Two-parameter Davis-Garsia decomposition}

We are going to apply Lemma \ref{westvirginia} to prove an extension of Theorem \ref{4summand} to square functions of double-indexed martingales.
\begin{thm} \label{2parDG} Let $\mathcal{F}_{i,j}$ be the canonical product filtration on $\Omega^\N\times \Omega^\N$, $f_{i,j}$ be adapted (and identified with functions on $\Omega^i\times \Omega^j$). Then
\begin{align} \label{eq:blueridgemountain}\E \sqrt{\sum_{i,j} \left|f_{i,j}\right|^2}
 \simeq \inf_{\substack{f_{i,j}=\alpha_{i,j}+\beta_{i,j}+\gamma_{i,j}+\delta_{i,j}\\ \alpha_{i,j},\beta_{i,j},\gamma_{i,j},\delta_{i,j}\ \mathrm{adapted}}}
 & \E\sum_{i,j} \left|\alpha_{i,j}\right|\\
 +& \E\sqrt{\sum_{i,j}\E_{i-1,j-1}\left|\beta_{i,j}\right|^2}\\
 +& \E\sum_i \sqrt{\sum_j \E_{\infty,j-1} \left|\gamma_{i,j}\right|^2}\\
 +& \E\sum_j \sqrt{\sum_i \E_{i-1,\infty} \left|\delta_{i,j}\right|^2}.
\end{align}
Moreover, if $f_{i,j}$ is a martingale difference sequence, the decomposition also can be chosen such that each summand is a martingale difference sequence.
\end{thm}
\begin{proof}
Without loss of generality we can assume that all but finitely many of $f_{i,j}$'s are 0, so they are $\mathcal{F}_{N,N}$-measurable and all summations can be taken over $0\leq i,j\leq N$ for some $N$. Since $\Omega$ can be modelled in $[0,1]$ with Lebesgue measure, by dyadic approximation we can assume that $\p{\Omega,\mathcal{F},\mu}$ is a finite set with normalized counting measure and keep in mind that the constants are supposed to not depend on the cardinality of $\Omega$. \par
The ``$\lesssim$'' inequality is not relevant for the main result and is left as an excercise. We will prove the ``$\gtrsim$'' inequality. Since the right hand side is a lattice norm of the sequence $f_{i,j}$, it is enough to have the inequality with infimum over 
\beq \label{eq:lattice}\left|\alpha_{i,j}+\beta_{i,j}+\gamma_{i,j}+\delta_{i,j}\right|\geq \left|f_{i,j}\right|\eeq
satisfied. By Corolary \ref{makeitdouble}, 
\beq \E\sqrt{\sum_{i,j}\left|f_{i,j}\right|^2}\gtrsim \int_{\p{\Omega^N}^4}\sqrt{\sum_{i,j} \left|f_{i,j}\p{\p{x_{<i},\eta_i},\p{y_{<j},\theta_j}}\right|^2}\d \eta\d \theta\d x\d y.\eeq
For each fixed $x,y$, we can apply Theorem \ref{4summand} to functions $\p{\eta_i,\theta_j}\mapsto f_{i,j}\p{\p{x_{<i},\eta_i},\p{y_{<j},\theta_j}}$, which results in functions $a_{i,j},b_{i,j},c_{i,j},d_{i,j}:\p{\Omega^\N \times \Omega}^2\to \{0,1\}$ such that 
\beq \label{eq:abcdsum1}a_{i,j}+b_{i,j}+c_{i,j}+d_{i,j}=1\eeq
and 
\begin{align}
\lefteqn{\int_{\p{\Omega^N}^4}\sqrt{\sum_{i,j} \left|f_{i,j}\p{\p{x_{<i},\eta_i},\p{y_{<j},\theta_j}}\right|^2}\d \eta\d \theta\d x\d y \gtrsim \label{eq:decdec}}\\
\int_{\p{\Omega^N}^2}\d x\d y &\left( \sum_{i,j}\int_{\Omega^2}\left|a_{i,j}\p{\p{x,\xi},\p{y,\upsilon} }f_{i,j}\p{\p{x_{<i},\xi},\p{y_{<j},\upsilon} }\right|\d \xi\d\upsilon \right. \label{eq:adec}\\
&+ \sqrt{\sum_{i,j}\int_{\Omega^2}\left|b_{i,j}\p{\p{x,\xi},\p{y,\upsilon} }f_{i,j}\p{\p{x_{<i},\xi},\p{y_{<j},\upsilon} }\right|^2\d \xi\d\upsilon} \label{eq:bdec}\\
&+ \sum_{i}\int_{\Omega }\d\xi \sqrt{\sum_j \int_{\Omega}\d \upsilon \left|c_{i,j}\p{\p{x,\xi},\p{y,\upsilon} }f_{i,j}\p{\p{x_{<i},\xi},\p{y_{<j},\upsilon} }\right|^2}\label{eq:cdec}\\
&\left. + \sum_{j}\int_{\Omega }\d\upsilon \sqrt{\sum_i \int_{\Omega}\d \xi \left|d_{i,j}\p{\p{x,\xi},\p{y,\upsilon} }f_{i,j}\p{\p{x_{<i},\xi},\p{y_{<j},\upsilon} }\right|^2}\right).\label{eq:ddec}
\end{align}
Care has to be taken concerning the notation, as in the above we are identifying $f_{i,j}$'s, which are $\mathcal{F}_{i,j}$-measurable functions on $\Omega^N\times \Omega^N$, with functions on $\Omega^i\times \Omega^j$, and then evaluating them at arguments which are not the original $\p{x_{\leq j},y_{\leq j}}$. Below, we will use $\E^{(x)}$ and $\E^{(x)}_i$ to indicate that we are taking $\E$ and $\E_i$ with respect to the variable $x$ and reserve $\E_i$ for undecoupled functions on $\Omega^N\times \Omega^N$. In particular, $f_{i,j}$ is in the image of $\E_{i,j}$, but $\p{\p{x,\eta},\p{y,\theta}}\mapsto f_{i,j}\p{\p{x_{<i},\eta},\p{y_{<j},\theta}}$ is in the image of $\E^{(x)}_{i-1}\E^{(y)}_{j-1}$. As usual, if a function depends only on first $i-1$ coordinates of $x$, we will write is as a function of $x_{<i}$. \par
We will prove the desired inequality \eqref{eq:blueridgemountain} with 
\beq \alpha_{i,j}(x,y)=\alpha_{i,j}\p{x_{\leq i},y_{\leq j}}= \mathbbm{1}_{\left\{\E^{(x,y)}_{i-1,j-1}a_{i,j}\geq \frac14\right\}}\p{\p{x_{<i},x_i},\p{y_{<j},y_j}} f_{i,j}(x,y),\eeq
\beq \beta_{i,j}(x,y)=\beta_{i,j}\p{x_{\leq i},y_{\leq j}}=\mathbbm{1}_{\left\{\E^{(x,y)}_{i-1,j-1}b_{i,j}\geq \frac14\right\}}\p{\p{x_{<i},x_i},\p{y_{<j},y_j}} f_{i,j}(x,y),\eeq
\beq \gamma_{i,j}(x,y)=\gamma_{i,j}\p{x_{\leq i},y_{\leq j}}=\mathbbm{1}_{\left\{\E^{(x,y)}_{i-1,j-1}c_{i,j}\geq \frac14\right\}}\p{\p{x_{<i},x_i},\p{y_{<j},y_j}} f_{i,j}(x,y),\eeq
\beq \delta_{i,j}(x,y)=\delta_{i,j}\p{x_{\leq i},y_{\leq j}}=\mathbbm{1}_{\left\{\E^{(x,y)}_{i-1,j-1}d_{i,j}\geq \frac14\right\}}\p{\p{x_{<i},x_i},\p{y_{<j},y_j}} f_{i,j}(x,y),\eeq
which clearly satisfy \eqref{eq:lattice} because of \eqref{eq:abcdsum1}. \par
In order to estimate \eqref{eq:adec}, for each $i,j$ we exploit $\mathcal{F}_{i-1,j-1}$-measurability of $f_{i,j}\p{\p{x_{<i},\xi},\p{y_{<j},\upsilon}}$ in $x,y$ and then treat $\xi,\upsilon$ as $x_i,y_j$. 
\begin{align}
&\sum_{i,j}\int_{\Omega^2}\d\xi\d\upsilon \int_{\p{\Omega^N}^2}\d x\d y\left|a_{i,j}\p{\p{x,\xi},\p{y,\upsilon}} f_{i,j}\p{\p{x_{<i},\xi},\p{y_{<j},\upsilon}}\right|\\
=& \sum_{i,j}\int_{\Omega^2}\d\xi\d\upsilon \int_{\p{\Omega^N}^2}\d x\d y \left| f_{i,j}\p{\p{x_{<i},\xi},\p{y_{<j},\upsilon}}\right|\E_{i-1,j-1}^{(x,y)}a_{i,j}\p{\p{x_{<i},\xi},\p{y_{<j},\upsilon}}\\
\gtrsim & \sum_{i,j}\int_{\Omega^2}\d\xi\d\upsilon \int_{\p{\Omega^N}^2}\d x\d y  \mathbbm{1}_{\left\{\E_{i-1,j-1}^{(x,y)}a_{i,j}\geq \frac14\right\}} \left| f_{i,j}\right| \p{\p{x_{<i},\xi},\p{y_{<j},\upsilon}}\\
=& \sum_{i,j}\int_{\Omega^2}\d\xi\d\upsilon \int_{\p{\Omega^N}^2}\d x\d y \left|\alpha_{i,j}\right|\p{\p{x_{<i},\xi},\p{y_{<j},\upsilon}}\label{eq:alphadec}\\
=&\sum_{i,j}\E \left|\alpha_{i,j}\right|.
\end{align}
Since we assumed that $\Omega$ is a finite set with counting measure, we apply Lemma \ref{westvirginia} with $I=[0,N]^2$, $J=\Omega^2$ to $f_{i,j}\p{\p{x_{<i},\xi},\p{y_{<j},\upsilon}}$, which are adapted to $\mathcal{F}_{i-1,j-1}$, and estimate \eqref{eq:bdec}. 
\begin{align}
&\int_{\p{\Omega^N}^2}\d x\d y \sqrt{\sum_{i,j}\int_{\Omega^2}\left|b_{i,j}\p{\p{x,\xi},\p{y,\upsilon} }f_{i,j}\p{\p{x_{<i},\xi},\p{y_{<j},\upsilon} }\right|^2\d \xi\d\upsilon} \\
\gtrsim & \int_{\p{\Omega^N}^2}\d x\d y \sqrt{\sum_{i,j}\int_{\Omega^2}\left| \mathbbm{1}_{\left\{\E_{i-1,j-1}^{(x,y)}b_{i,j}\geq \frac14\right\}}  f_{i,j}\p{\p{x_{<i},\xi},\p{y_{<j},\upsilon}} \right|^2\d \xi\d\upsilon}\\
=& \int_{\p{\Omega^N}^2}\d x\d y \sqrt{\sum_{i,j} \int_{\Omega^2} \d \xi\d \upsilon \left|\beta_{i,j}\p{\p{x_{<i},\xi},\p{y_{<j},\upsilon}}\right|^2}\label{eq:betadec}\\
=& \E \sqrt{\sum_{i,j} \E_{i-1,j-1}\left|\beta_{i,j}\right|^2}. 
\end{align}
Now we estimate \eqref{eq:cdec}. First, we apply convexity of norm and bring $\E_{i-1}^{(x)}$ inside. Then, for fixed $i,\xi$ and $x$, we apply Lemma \ref{westvirginia} with $I=[0,N]$, $J=\Omega$ and filtration $\mathcal{F}_{\infty, j-1}$. 
\begin{align}
& \sum_i \int_{\Omega}\d \xi \int_{\p{\Omega^N}^2}\d x\d y \sqrt{\sum_j \int_{\Omega}\d \upsilon \left|c_{i,j}\p{\p{x,\xi},\p{y,\upsilon} }f_{i,j}\p{\p{x_{<i},\xi},\p{y_{<j},\upsilon} }\right|^2}\\
\geq & \sum_i \int_{\Omega}\d \xi \int_{\p{\Omega^N}^2}\d x\d y \sqrt{\sum_j \int_{\Omega}\d \upsilon \left|f_{i,j}\E^{(x)}_{i-1} c_{i,j}\p{\p{x_{<i},\xi},\p{y_{<j},\upsilon}}\right|^2} \\
\gtrsim & \sum_i \int_{\Omega}\d \xi \int_{\p{\Omega^N}^2}\d x\d y \sqrt{\sum_j \int_{\Omega}\d \upsilon \left|\mathbbm{1}_{\left\{\E^{(x,y)}_{i-1,j-1}c_{i,j}\geq \frac14\right\}}f_{i,j}\p{\p{x_{<i},\xi},\p{y_{<j},\upsilon}}\right|^2}\\
=& \sum_i \int_{\Omega}\d \xi \int_{\p{\Omega^N}^2}\d x\d y \sqrt{\sum_j \int_{\Omega}\d \upsilon \left|\gamma_{i,j}\p{\p{x_{<i},\xi},\p{y_{<j},\upsilon}}\right|^2}\label{eq:gammadec}\\
=& \sum_i \E\sqrt{\sum_j \E_{\infty,j-1}\left|\gamma_{i,j}\right|^2}.
\end{align}
Analogously, 
\begin{align} & \sum_j \int_{\Omega}\d \upsilon \int_{\p{\Omega^N}^2}\d x\d y \sqrt{\sum_i \int_{\Omega}\d \xi\left|d_{i,j}\p{\p{x,\xi},\p{y,\upsilon} }f_{i,j}\p{\p{x_{<i},\xi},\p{y_{<j},\upsilon} }\right|^2}\\
\gtrsim & \sum_j \int_{\Omega}\d \upsilon \int_{\p{\Omega^N}^2}\d x\d y \sqrt{\sum_i \int_{\Omega}\d \xi \left|\delta_{i,j}\p{\p{x_{<i},\xi},\p{y_{<j},\upsilon}}\right|^2}\label{eq:deltadec}\\
= &\sum_j \E\sqrt{\sum_i \E_{i-1,\infty}\left|\delta_{i,j}\right|^2}.
\end{align}
Adding the resulting inequalities and plugging into \eqref{eq:decdec} gives the desired inequality. \par
If $f_{i,j}$ form a martingale difference sequence, we take a decomposition 
\beq f_{i,j}=\alpha_{i,j}+\beta_{i,j}+\gamma_{i,j}+\delta_{i,j}\eeq
verifying \eqref{eq:blueridgemountain} and apply $\Delta_{i,j}=\Delta_i\otimes \Delta_j$ (which is the projection onto mean zero dependence in the last components of both variables) to both sides of this equality. This only decreases, up to the constant $4$, each summand of the right hand side of \eqref{eq:blueridgemountain}, because, as can be seen through \eqref{eq:alphadec}, \eqref{eq:betadec}, \eqref{eq:gammadec}, \eqref{eq:deltadec}, it is just projection onto mean zero dependence in $\xi$ and $\upsilon$.  
\end{proof}

\begin{thm}Theorem \ref{2parDG} is true for filtrations defined by \eqref{eq:univdefi}, \eqref{eq:univdefi}. \end{thm}
\begin{proof}[Proof (sketch).]
Without losss of generality, $\Omega_{i,j}$ can be the same finite probability space. We decouple the $L_1\p{\ell_2}$ norm on the left side analogously as previously, namely
\beq \int_{\Omega^{[1,N]\times [1,N]}}\d x\p{\sum_{i,j}\left|f_{i,j}\p{x_{[1,i]\times [1,j]}}\right|^2}^\frac12 \simeq\eeq
\beq  \int_{\p{\Omega^{[1,N]\times [1,N]}}^4}\d x\d y\d t\d z\p{\sum_{i,j}\left|f_{i,j}\p{x_{[1,i-1]\times [1,j-1]}, y_{ \{i\}\times [1,j-1]}, z_{ [1,i-1]\times \{j\}}, t_{ \{(i,j)\}} }\right|^2}^\frac12.
\eeq
Then, we proceed as previously, by iterating the procedure of decomposing with respect to one variable and then enforcing the decomposition to be of appropriate measurability with respect to the ones yet not used. Namely, we start with $t$, then correct the weight with respect to $y$ by first ensuring $w_{i,j}$ to depend on $y_{\{i\}\times [1,N]}$, then on $y_{\{i\}\times [1,j-1]}$, then same for $z$. Then we use the four summand decomposition with respect to the variables $y_{\{i\}\times [1,N]}$ and $z_{[1,N]\times\{j\}}$, correct it to have desired dependence on $y,z$ and ultimately correct all the weights with respect to $x$. 
\end{proof}
Now we are ready to prove the main result. 
\begin{thm}For Hardy spaces on the canonical double indexed product filtration,
\beq \|F\|_{H_1^S}\gtrsim \|F\|_{H_1^*}. \eeq
\end{thm}
\begin{proof}
We apply Theorem \ref{2parDG} to $f_{i,j}=\Delta_{i,j}F$. This results in a decomposition 
\beq \Delta_{i,j}F= \Delta_{i,j}A+ \Delta_{i,j}B+ \Delta_{i,j}C+ \Delta_{i,j}D\eeq
of $\Delta_{i,j}F$ into a sum of martingale difference sequences satisfying
\begin{align} \|F\|_{H_1^S}\gtrsim & \E\sum_{i,j}\left|\Delta_{i,j}A\right|\\
+&  \E\sqrt{\sum_{i,j} \E_{i-1,j-1}\left|\Delta_{i,j} B\right|^2}\\
+& \E\sum_i \sqrt{\sum_j \E_{\infty,j-1}\left|\Delta_{i,j}C\right|^2}\\
+& \E\sum_j \sqrt{\sum_i \E_{i-1,\infty}\left|\Delta_{i,j}D\right|^2}.
\end{align}
By triangle inequality,
\beq \|F\|_{H_1^*}\leq \|A\|_{H_1^*}+\|B\|_{H_1^*}+\|C\|_{H_1^*}+\|D\|_{H_1^*}.\eeq
Trivially, 
\beq \|A\|_{H_1^*}\leq \E\sum_{i,j}\left|\Delta_{i,j}A\right|.\eeq
The inequality 
\beq \|B\|_{H_1^*}\lesssim \E\sqrt{\sum_{i,j} \E_{i-1,j-1}\left|\Delta_{i,j} B\right|^2}\eeq
is an inequality due to Weisz \cite{niestety}. 
By the one parameter result $\|\cdot\|_{H_1^*}\lesssim \|\cdot\|_{H_1^s}$ applied at each $i$ to $\E_{i,\infty}C-\E_{i-1,\infty}C$,
\begin{align}
\E\sup_{i,j} \left|\E_{i,j}C\right|\leq & \E\sup_j \sum_i \left|\E_{\infty,j}\p{\E_{i,\infty}-\E_{i-1,\infty}}C\right|\\
\leq & \E \sum_i \sup_j \left|\E_{\infty,j}\p{\E_{i,\infty}-\E_{i-1,\infty}}C\right|\\
= & \sum_i \left\| \p{\E_{i,\infty}-\E_{i-1,\infty}}C\right\|_{H_1^*\left[\p{\mathcal{F}_{\infty,j}}_j\right]}\\
\lesssim & \sum_i \left\| \p{\E_{i,\infty}-\E_{i-1,\infty}}C\right\|_{H_1^s\left[\p{\mathcal{F}_{\infty,j}}_j\right]}\\
=& \E\sum_i \sqrt{\sum_j \E_{\infty,j-1}\left|\Delta_{i,j}C\right|^2}.
\end{align}
The inequality for $D$ is analogous and we are done. 
\end{proof}
\section{Attempts to attack the other inequality}
Our proof of Lemma \ref{westvirginia} can be factorized into 2 parts: the Lagrange multiplier part and a proof that some functions sequence is bounded away from zero in the norm dual to the right-hand side. The first part can be phrased in a separate lemma which we borrow from \cite{rzjsch}.
\begin{lem}\label{gradlemma} Let $N\in \N$, $\p{\Omega,\mathcal{F},\mu}$ be a finite probability space, $V_k$ be subspaces of $L^2\p{\Omega}$, $V$ be the subspace of $L^2\p{\Omega,\ell^2_N}$ consisting of sequences $\p{f_{n}}_{n=1}^N$ of functions such that $f_n\in V_n$, $P_{V_k}$, $P_V$ be orthogonal projections onto $V_k$, $V$ respectively, $\|\cdot\|_X$ be a random norm on $\R^N$ differentiable outside of $0$, $\|\cdot\|_Y$ be a norm on $V$, $Y^*$ be the dual norm on $V$ in the sense of the usual pairing. Then for any constants $C>0$, $q> 1$, the following are equivalent:\\
(i) for any $\p{f_1,\ldots,f_N}\in V$,
\beq \label{eq:lemgradi}\E \left\|\p{f_n}_{n=1}^N\right\|_X^q\geq C\left\|\p{f_n}_{n=1}^N\right\|_Y^q\eeq
(ii) for any $\p{f_1,\ldots,f_N}\in V$ not identically zero,
\beq \label{eq:lemgradii}\left\| P_V\p{\|f\|_X^{q-1} \nabla\| \cdot \|_X (f)}\right\|_{Y^*}\geq C \|f\|_Y^{q-1},\eeq
where $\nabla\|\cdot\|_X$ is extended to be equal to $0$ in $0$. Moreover, if for any $\p{f_1,\ldots,f_N}\in V$ not identically zero,
\beq \label{eq:weirdcond}\left\| P_{V\cap F}\p{\nabla\|\cdot\|_X(f)}\right\|_{Y^*}\geq C\left\|P_{V\cap F}\right\|_{Y^*\to Y^*},\eeq
where $F=\left\{\varphi\in V: \supp \varphi_i\subset \supp f_i\right\}$, then 
\beq \label{eq:l1xgeqcy}\E\|f\|_X\geq C\|f\|_Y\eeq 
for all $f\in V$. 
\end{lem}

Suppose we know that any martingale with repsect to an (F4) filtration can be approximated by a martingale with repsect to an (F4) filtration on a finite probability space (in particular, this would be true if the embedding theorem from \cite{RzT} was true without the discreteness assumption). Our inequality can be written as 
\beq \E \p{\sum_{n\leq N} \left|\sum_{k\leq n}d_k\right|^p}^\frac{1}{p}\geq C_{p,N} \left\|\sum_k d_k\right\|_{\p{H_1^S}^*}\eeq
where $p\to \infty$, $N\in \mathbb{N}^2$, and $d_k$ is a $k$-th martingale difference. We may apply the above lemma with $q=1$, $V_k$ being the space of $k$-th martingale difference, $X=\ell_p$, $Y=H_1^S$. One easily checks that 

By calculating th gradient we arrive at
\beq \left\| \sum_k \Delta_k \frac{ \sum_{n\geq k}\sgn f_n \left|f_n\right|^{p-1} }{ \p{\sum_j \left|f_j\right|^p}^\frac{p-1}{p} }\right\|_{ \p{H_1^S}^* }\geq C_{p,N}.\eeq

\end{document}